\documentclass[12pt]{article}
\usepackage{hyperref}
\usepackage{cite}
\usepackage[utf8]{inputenc}
\usepackage[english]{babel}
\usepackage{amsmath}
\usepackage{amsfonts}
\usepackage{amssymb}
\usepackage{amsthm}
\usepackage{xypic}
\author{Jaime Castro P\'erez\footnote{jcastrop@itesm.mx} \\ \textit{Escuela de Ingenier\'ia y Ciencias, Instituto Tecnol\'ologico y de} \\ \textit{Estudios Superiores de Monterrey} \\ \textit{Calle del Puente 222, Tlalpan, 14380, M\'exico D.F., M\'exico.} \\ Mauricio Medina B\'arcenas\footnote{mmedina@matem.unam.mx} \; Jos\'e R\'ios Montes\footnote{jrios@matem.unam.mx} \\  Angel Zald\'ivar\footnote{zaldivar@matem.unam.mx}\\ \textit{Instituto de Matem\'aticas, Universidad Nacional} \\ \textit{Aut\'onoma de M\'exico} \\ \textit{Area de la Investigaci\'on Cient\'ifica, Circuito Exterior, C.U.,} \\ \textit{04510, M\'exico D.F., M\'exico.}}
\title{Boyle's Conjecture and perfect localizations}

\theoremstyle{plain}
\newtheorem{thm}{\protect\theoremname}[section]
  \theoremstyle{definition}
  \newtheorem{rem}[thm]{\protect\remarkname}
 \theoremstyle{definition}
  
\theoremstyle{definition}
  
  \theoremstyle{plain}
  \newtheorem{cor}[thm]{\protect\corollaryname}
  \theoremstyle{plain}
  \newtheorem{lemma}[thm]{\protect\lemmaname}
  \theoremstyle{plain}
  \newtheorem{prop}[thm]{\protect\propositionname}
    \theoremstyle{definition}
  \newtheorem{dfn}[thm]{\protect\definitionname}

\makeatother

\providecommand{\lemmaname}{Lemma}
\providecommand{\corollaryname}{Corollary}
\providecommand{\propositionname}{Proposition}
\providecommand{\definitionname}{Definition}
\providecommand{\theoremname}{Theorem}
\providecommand{\remarkname}{Remark}
\providecommand{\examplename}{Example}

\newcommand{\Hom}{Hom_R}

\begin{document}
\maketitle

\begin{abstract}
In this article we study the behaviour of left QI-rings under perfect localizations. We show that a perfect localization of a left QI-ring is a left QI-ring. We prove that Boyle's conjecture is true for left QI-rings with finite Gabriel dimension such that the hereditary torsion theory generated by semisimple modules is perfect. As corollary we get that Boyle's conjecture is true for left QI-rings which satisfy the restricted left socle condition, this result was proved first by C. Faith in \cite{faithhereditary}.
\end{abstract}

\section{Introduction}
Through all this paper $R$ will denote an associative ring with unit element. We will work with unitary left $R$-modules and the category of modules will be denoted by $R$-Mod. For general background on module and ring theory we refer the reader to \cite{andersonrings}, \cite{kaschmodules}, \cite{lamlectures} and \cite{stenstromrings}. 

Remember that an $R$-module $M$ is quasi-injective if every morphism $f:N\to M$ with $N\leq M$ can be extended to an endomorphism of $M$. Equivalently $M$ is quasi-injective if and only if $M$ is fully invariant in its injective hull. A ring $R$ is called left QI-ring if every quasi-injective left $R$-module is injective, these rings were introduced by A. Boyle in \cite{boylehereditary}.  Also in \cite{boylehereditary} is shown that a left QI-ring is left noetherian and left $V$-ring; recall that a left $V$-ring is defined as a ring where every simple left module is injective. 

In \cite{cozzenshomological} the author introduces two examples of non semisimples left QI-rings . These examples are left hereditary rings, that is, every left ideal is projective. 

In \cite[Theorem 5]{boylehereditary} A. Boyle characterizes two-sided hereditary, right noetherian, left QI-rings and she conjectured that every left QI-ring is left hereditary. In \cite{faithhereditary} C. Faith gave an approach to this conjecture. In \cite[Corrolary 3]{faithhereditary} is shown that every left QI-ring is a finite product of simple left QI-rings, so it is enough formulate Boyle's conjecture for simple left QI-rings. 

C. Faith, in \cite[Theorem 18]{faithhereditary}, answers in affirmative Boyle's conjecture for simple left QI-rings which satisfy  the restricted left socle condition (RLS):

\begin{center}
If $I\neq R$ is an essential left ideal, then $R/I$ has non zero socle.
\end{center}

The Theorem 18 in \cite{faithhereditary} was extended by T. Rizvi and D. Van Huynh in \cite{rizvianapproach}. But the conjecture is still open. 		

This paper is organized in three sections, the first one is this introduction and the second one concerns to present the necessary preliminaries. 

Section 3 is where we develop our work and it contains the main results, we prove that if $R$ is a left QI-ring and $\tau$ is a perfect torsion theory then the ring of quotients $_\tau R$ is left QI (Proposition \ref{loc}). Also, we prove that Boyle's conjecture is true for all those left QI-rings with finite Gabriel dimension such that the hereditary torsion theory generated by the semisimple modules is perfect (Theorem \ref{ssperf}).

\section{Preliminaries}

One useful tool to characterize rings, is the hereditary torsion theories. Let us recall the definition of hereditary torsion theory:

\begin{dfn}
A pair of nonempty classes of modules $\tau=(\mathfrak{T},\mathfrak{F})$ is a torsion theory if 
\begin{enumerate}
\item $\Hom(T,F)=0$ for all $T\in\mathfrak{T}$ and $F\in\mathfrak{F}$.
\item If $M$ is such that $\Hom(M,F)=0$ for all $F\in\mathfrak{F}$ then $M\in\mathfrak{T}$.
\item If $N$ is such that $\Hom(T,N)=0$ for all $T\in\mathfrak{T}$ then $N\in\mathfrak{F}$.
\end{enumerate}
It is said $\tau$ is hereditary if $\mathfrak{T}$ is closed under submodules. The set of hereditary tosion theories in $R$-Mod is denoted $R$-tors. The class $\mathfrak{T}$ is called the torsion class and $\mathfrak{F}$ the torsion free class.
\end{dfn}

It can be proved $R$-tors is a frame with order the inclusion of torsion classes, the least and greatest elements of $R$-tors are denoted by $\xi$ and $\chi$ respectively. Given a class of modules $\mathcal{C}$ there exists the least hereditary torsion theory $\xi(\mathcal{C})$ such that $\mathcal{C}$ is contained in the torsion class and there exist the greatest hereditary torsion theory $\chi(\mathcal{C})$ such that $\mathcal{C}$ is contained in the torsion free class. If $\tau=(\mathfrak{T},\mathfrak{F})$, the modules in $\mathfrak{T}$ are called $\tau$-torsion modules and the modules in $\mathfrak{F}$ are called $\tau$-torsion free. For more details see \cite{golantorsion}. 

There exist a bijective correspondence between $R$-tors and $R$-Gab, where $R$-Gab denotes the set of Gabriel topologies in $R$. see \cite[VI, Theorem 5.1]{stenstromrings}. 

An $R$-module $N$ is called $\tau$-cocritical with $\tau\in R$-tors if $N$ is $\tau$-torsion free but every proper factor module is $\tau$-torsion. With this modules can be defined the Gabriel filtration in $R$-tors as:
\[\tau_0=\xi\]
If $i$ is a non limit ordinal:
\[\tau_i=\tau_{i-1}\vee\bigvee\{\xi(N)|N\;\text{is}\;\tau_{i-1}\text{-cocritical}\}\]
If $i$ is a limit ordinal:
\[\tau_i=\bigvee_{j< i}\tau_j\]

Since $R$-tors is a set it must exists the least ordinal $\alpha$ such that $\tau_{\alpha}=\tau_{\alpha+\beta}$ for all ordinals $\beta$. If $\tau_\alpha=\chi$ then we say $R$ has Gabriel dimension equal to $\alpha$ and we denote it as $Gdim(R)=\alpha$.

The concept of QI-ring can be generalized to modules, one paper in this sense is \cite{daunsqi}. In that paper it is shown (\cite[Theorem 3.9]{daunsqi}) that (in a more general context) a ring $R$ is left QI if and only if $R$ has Gabriel dimension and every hereditary pretorsion class is a hereditary torsion class. A hereditary pretorsion class in  $R$-Mod is a class of modules closed under submodules, direct sums and quotients.

Let us recall the concept of perfect localization. 

\begin{dfn}
If $\varphi:R\to S$ is an epimorphism in the category of rings which makes $S$ into a flat right $R$-module, then we will call $S$ a left perfect  localization of $R$.
\end{dfn}

\begin{rem}\label{perf}
Given a hereditary torsion theory $\tau\in R$-tors we will say that $\tau$ is perfect if the ring of quotients $_\tau R$ is a perfect left localization of $R$.
\end{rem}

\begin{rem}\label{ex}
Given $\tau\in R$-tors, let us denote the localization functor as $\mathfrak{q}_\tau:R-Mod\to {_\tau R}-Mod$. If $\tau$ is perfect then $\mathfrak{q}_\tau$ is exact. \cite[XI, Proposition 3.4]{stenstromrings}. We will write $_\tau N=\mathfrak{q}_\tau(N)$.
\end{rem}

\section{Left QI-rings and perfect torsion theories}

\begin{rem}\label{free}
Let $\tau\leq\sigma\in R$-tors. Note that if $N$ is $\sigma$-torsion free then $\mathfrak{q}_\tau(N)$ is $\sigma$-torsion free. In fact, since $\tau\leq\sigma$ then $N$ is $\tau$-torsion free, so we have an essential monomorphism $\psi_N:N\to\mathfrak{q}_\tau(N)$. 
\end{rem}

\begin{lemma}\label{tt}
Let $\tau\leq\sigma\neq\chi\in R$-tors perfect torsion theories with $R$ $\sigma$-torsion free. Let $\mathfrak{q}_\tau:R-Mod\to {_\tau R}-Mod$ denote the localization functor.  
\begin{enumerate}
\item If $\sigma=(\mathfrak{T},\mathfrak{F})$ then $\hat{\sigma}:=(\mathfrak{q}_\tau(\mathfrak{T}),\mathfrak{q}_\tau(\mathfrak{F}))\in {_\tau{R}}$-tors.
\item If $\mathcal{F}_\sigma$ and $\mathcal{F}_{\hat{\sigma}}$ denote the Gabriel topologies in $R$ and $_\tau R$ associated to $\sigma$ and $\hat{\sigma}$ respectively, then:
\[J\in\mathcal{F}_{\hat{\sigma}}\Leftrightarrow J={_\tau I}\;\text{with}\;I\in\mathcal{F}_\sigma\]
\item $_\sigma R$ is $\hat{\sigma}$-closed (as $_\tau R$-module).
\item There is a ring isomorphism $_\sigma R\cong {_{\hat{\sigma}}(_\tau R)}$
\item $\hat{\sigma}\in {_\tau R}$-tors is perfect.

\end{enumerate}
\end{lemma}

\begin{proof}
\textit{1}. Let $\tau\leq\sigma\in R$-tors, with $\sigma=(\mathfrak{T},\mathfrak{F})$. Then $\hat{\sigma}=(\mathfrak{q}_\tau(\mathfrak{T}),\mathfrak{q}_\tau(\mathfrak{F}))$ where
\[\mathfrak{q}_\tau(\mathfrak{T})=\{\mathfrak{q}_\tau(M)|M\in\mathfrak{T}\}\]
\[\mathfrak{q}_\tau(\mathfrak{F})=\{\mathfrak{q}_\tau(N)|N\in\mathfrak{F}\}\]

Let $\mathfrak{q}_\tau(M)\in\mathfrak{q}_\tau(\mathfrak{T})$ and $\mathfrak{q}_\tau(N)\in\mathfrak{q}_\tau(\mathfrak{F})$. Since $\tau$ is perfect, $\mathfrak{q}_\tau(M)={_\tau R}\otimes_R M$. Then
\[Hom_{_\tau R}(\mathfrak{q}_\tau(M),{_\tau N})=Hom_{_\tau R}({_\tau R}\otimes_RM,{_\tau N})\cong Hom_R(M,Hom_{_\tau R}({_\tau R},{_\tau N}))\]
\[\cong Hom_R(M,{_\tau N})=0\]
by remark \ref{free}. 

Now, let $\mathfrak{q}_\tau(N)={_\tau N}$ an $_\tau R$-module such that $Hom_{_\tau R}({_\tau R\otimes_R M},{_\tau N})=0$ for all $M\in\mathfrak{T}$. Following the above isomorphisms, we get that $Hom_R(M,{_\tau(N)})=0$, i.e., $_\tau N\in\mathfrak{F}$. Since $\tau\leq\sigma$, we have a monomorphism $\psi_N:N\to {_\tau N}$ so $N\in\mathfrak{F}$. Thus $_\tau N\in\mathfrak{q}_\tau(\mathfrak{F})$.

On the other hand, suppose $_\tau M$ is an $_\tau R$-module such that $Hom_{_\tau R}({_\tau M}, {_\tau N})=0$ for all ${_\tau N}\in\mathfrak{F}$. Then, we have that $Hom_R(M,{_\tau N})=0$. Let $N\in\mathfrak{F}$ and suppose $Hom_R(M,N)\neq 0$. Hence there exists $0\neq f:M\to N$, this implies that $0\neq\psi_Nf:M\to {_\tau N}$. Contradiction. Thus $Hom_R(M,N)=0$ for all $N\in\mathfrak{F}$, hence $M\in\mathfrak{T}$. 

Thus, we have that $\hat{\sigma}$ is a torsion theory. Let us see it is hereditary.

Let $_\tau M\in\mathfrak{q}_\tau(\mathfrak{T})$ and $K\leq {_\tau M}$ an $_\tau R$-submodule. There is a monomorphism $\psi:M/\tau(M)\to {_\tau M}$. Consider $K\cap (M/\tau(M))$ which is a $\sigma$-torsion $R$-module. By \cite[XI,Proposition 3.7]{stenstromrings} $\mathfrak{q}_\tau(K\cap(M/\tau(M)))=K$. Thus $K\in\mathfrak{q}_\tau(\mathfrak{T})$.

\textit{2}$\Leftarrow$. Let $I\in\mathcal{F}_\sigma$. We have the following commutative diagram with exact rows:
\[\xymatrix{0\ar[r] & I\ar[r]\ar[d] & R\ar[r]\ar[d] & R/I \ar[r]\ar[d] & 0 \\ 
0\ar[r] & {_\tau I}\ar[r] & {_\tau R}\ar[r] & {_\tau(R/I)}\ar[r] & 0}\]

Since $R/I$ is $\sigma$-torsion then ${_\tau(R/I)}\in\mathfrak{q}_\tau(\mathfrak{T})$ and ${_\tau R}/{_\tau I}\cong {_\tau(R/I)}$. Thus ${_\tau I}\in\mathcal{F}_{\hat{\sigma}}$.

$\Rightarrow$. Let $E$ be an injective $R$-module such that $\chi(E)=\sigma$. Since $E$ is $\sigma$-torsion free then it is $\tau$-torsion free. Hence $E$ is $\tau$-closed, so $E={_\tau}E$. Let $J\in\mathcal{F}_{\hat{\sigma}}$, since $R$ is $\tau$-torsion free then $J={_\tau(J\cap R)}$. Therefore:
\[\Hom(R/J\cap R,E)=\Hom(R/J\cap R, {_\tau E})\cong\Hom(R/J\cap R,Hom_{_\tau R}({_\tau R, E}))\]
\[\cong Hom_{_\tau R}({_\tau R}\otimes_RR/J\cap R, E)\cong Hom_{_\tau R}({_\tau(R/J\cap R)}, E)=Hom_{_\tau R}({_\tau R}/J,E)\]

Since $E$ is $\sigma$-torsion free then $E$ is $\hat{\sigma}$-torsion free but ${_\tau R}/J$ is $\hat{\sigma}$-torsion. Thus $Hom_{_\tau R}({_\tau R}/J,E)=0$. This implies $\Hom(R/J\cap R,E)=0$ and hence $J\cap R\in\mathcal{F}_\sigma$.

\textit{3}. Let us see first that $_\sigma R$ is $\tau$-closed (as $R$-module). By \cite[IX, Proposition 1.8]{stenstromrings} $_\sigma R$ is $\sigma$-closed, that is, $_\sigma R$ is $\sigma$-torsion free and $\mathcal{F}_\sigma$-injective. Thus $_\sigma R$ is $\tau$-torsion free and since $\mathcal{F}_\tau\subseteq\mathcal{F}_\sigma$ then $_\sigma R$ is $\mathcal{F}_\tau $-injective.

Since $R$ is $\sigma$-torsion free then $_\tau R$ is $\hat{\sigma}$-torsion free and we have following inclusions 
\[R\hookrightarrow {_\tau R}\hookrightarrow {_\sigma R}\]

Also, we have to note that this inclusions are essential, so $_\sigma R$ is $\hat{\sigma}$-torsion free. Now, let $J\in\mathcal{F}_{\hat{\sigma}}$ and  $g:J\to {_\sigma}R$ an $_\tau R$-morphism. Then $J={_\tau(J\cap R)}$ and by (\textit{2}) $R/J\cap R$ is $\sigma$-torsion. If $\psi:R/J\cap R\to {_\tau(R/J\cap R)}={_\tau R/J}$ is the canonical homomorphism then $Ker(\psi)$ and $Coker(\psi)$ are $\tau$-torsion, therefore they are $\sigma$-torsion. So we have the following short exact sequence:
\[0\to \frac{R/J\cap R}{Ker(\psi)}\to {_\tau R/J}\to Coker(\psi)\to 0\]
Thus, $_\tau R/J$ is $\sigma$-torsion.

Since $_\sigma R$ is $\mathcal{F}_\sigma$-injective there exists an $R$-morphism $\bar{g}:{_\tau R}\to {_\sigma R}$ such that $\bar{g}|_J=g$. We have that $_\sigma R$ is $\tau$-closed, so $\bar{g}$ is an $R$-morphism between $\tau$-closed $R$-module, hence $\bar{g}$ is an $_\tau R$-morphism. Thus $_\sigma R$ is $\mathcal{F}_{\hat{\sigma}}$-injective.

\textit{4}. Since $_\sigma R$ is an $_\tau R$-module which is $\hat{\sigma}$-closed and $_\tau R\leq_e {_\sigma R}$ then $_\sigma R\cong {_{\hat{\sigma}}(_\tau R)}$.

\textit{5}. We have $_\sigma R\cong {_{\hat{\sigma}}}({_\tau R})$. Let $N$ be an $_\sigma R$-module. Since $\sigma$ is perfect then $N={_\sigma N}$ with $_RN$ $\sigma$-torsion free. Then $_\tau N\in\mathfrak{q}_\tau(\mathfrak{F})$. Thus $_\sigma N$ is $\hat{\sigma}$-torsion free. By \cite[XI, Ex. 6]{stenstromrings} $\hat{\sigma}\in {_\tau R}$-tors is perfect. 

\end{proof}

\begin{lemma}\label{cocrit}
Let $\tau\leq\sigma$ be perfect torsion theories in $R$-Mod and let $\mathfrak{q}_\tau:R-Mod\to {_\tau R}-Mod$ the localization functor.
\begin{enumerate}
\item If $M$ is $\sigma$-cocritical then $_\tau M$ is $\mathfrak{q}_\tau(\sigma)$-cocritical.
\item If an $_\tau R$-module $K$ is $\mathfrak{q}_\tau(\sigma)$-cocritical then $K$ as $R$-module is $\sigma$-cocritical.
\end{enumerate}
\end{lemma}

\begin{proof}
\textit{1}. Since $M$ is $\sigma$-torsion free then $_\tau M$ is $\mathfrak{q}_\tau(\sigma)$-torsion free. Let $N\leq {_\tau M}$ be an $_\tau R$-submodule. Since $M$ is $\sigma$-torsion free then it is $\tau$-torsion free, so the canonical morphism $\psi_M:M\to {_\tau M}$ is a monomorphism. By \cite[IX, Proposition 4.3]{stenstromrings} $N={_\tau(N\cap M)}$. Since $\tau$ is perfect, then $\mathfrak{q}_\tau$ is exact, hence $\frac{_\tau M}{N}\cong{_\tau(\frac{M}{N\cap M})}$. By hypothesis $\frac{M}{M\cap N}$ is $\sigma$-torsion thus $\frac{_\tau M}{N}$ is $\mathfrak{q}_\tau(\sigma)$-torsion.

\textit{2}. Let $K$ be a $\mathfrak{q}_\tau(\sigma)$-cocritical. Since $K$ is $\mathfrak{q}_\tau(\sigma)$-torsion free then $K={_\tau M}$ for some $M$ $\sigma$-torsion free. So, as $R$-module $K$ is $\sigma$-torsion free. 
Now, let $L< K$ be an $R$-submodule such that $K/L$ is $\sigma$-torsion free. Since $\mathfrak{q}_\tau$ is exact $\frac{K}{_\tau L}\cong{_\tau(\frac{K}{L})}$ but ${_\tau(\frac{K}{L})}$ is $\mathfrak{q}_\tau(\sigma)$-torsion free and $\frac{K}{_\tau L}$ $\mathfrak{q}_\tau(\sigma)$-torsion, this is a contradiction. This implies that $N/L=t_\sigma(K/L)\neq 0$ and $\frac{K}{N}\cong\frac{K/L}{N/L}$ is $\sigma$-torsion free, hence $N=K$. Thus $K/L$ is $\sigma$-torsion.
\end{proof}


\begin{lemma}\label{gab}
Let $R$ be a ring with finite Gabriel dimension, $Gdim(R)=n$. Let $\{\tau_i\}_{i=0}^n$ be the Gabriel filtration in $R$-tors with every $\tau_i$ perfect. If $\mathfrak{q}_{\tau_1}:R-Mod\to {_{\tau_1}R-Mod}$ is the localization functor and $\{\omega_j\}$ is the Gabriel filtration in $_{\tau_1}R$-tors, then $\mathfrak{q}_{\tau_1}(\tau_{i+1})=\omega_i$ for all $0\leq i$.
\end{lemma}

\begin{proof}

By induction over $i$. If $i=0$ then $\mathfrak{\tau_1}=\xi\in{_{\tau_1}}R$-tors and $\omega_0=\xi$. Now suppose the result is valid for each natural less than $i$. 

By hypothesis of induction $\mathfrak{q}_{\tau_1}(\tau_{i})=\omega_{i-1}$, so
\[\omega_{i}=\omega_{i-1}\vee\bigvee\{\xi(K)|K\;\text{is}\;\omega_{i-1}\text{-cocritical}\}\]
\[=\mathfrak{q}_{\tau_1}(\tau_{i})\vee\bigvee\{\xi(K)|K\;\text{is}\;\mathfrak{q}_{\tau_1}(\tau_{i})\text{-cocritical}\}\]

Let $K$ be a $\mathfrak{q}_{\tau_1}(\tau_{i})$-cocritical, then by Lemma \ref{cocrit}.\textit{2} $K$ as $R$-module is $\tau_{i}$-cocritical, hence $K$ is $\tau_{i+1}$-torsion. Thus $K$ is $\mathfrak{q}_{\tau_1}(\tau_{i+1})$-torsion. Therefore $\omega_i\leq\mathfrak{q}_{\tau_1}(\tau_{i+1})$. By Lemma \ref{cocrit}.\textit{1} if $N$ is $\tau_i$-cocritical then $_{\tau_1}N$ is $\mathfrak{q}_{\tau_1}(\tau_i)=\omega_{i-1}$-cocritical, so $_{\tau_1}N$ is $\omega_i$-torsion.

Let $L$ be a $\omega_i$-torsion free. Then $L$ is $\omega_{i-1}=\mathfrak{q}_{\tau_1}(\tau_i)$-torsion free and hence $_RL$ is $\tau_i$-torsion free. If $L$ is not $\tau_{i+1}$-torsion free there exists an $R$-morphism $0\neq f:N\to L$ with $N$ $\tau_i$-cocritical. Then we have a non zero $_{\tau_1}R$-morphism $_{\tau_1}f:{_{\tau_1}}N\to L$ with $_{\tau_1}N$ $\omega_{i-1}$-cocritical. Contradiction. Thus $L$ is $\tau_{i+1}$-torsion free. This implies that $L$ is $\mathfrak{q}_{\tau_1}(\tau_{i+1})$-torsion free. Thus $\mathfrak{q}_{\tau_1}(\tau_{i+1})\leq\omega_i$.

%
\end{proof}

\begin{cor}\label{gab2}
Let $R$ be a ring with finite Gabriel dimension, $Gdim(R)=n$. Let $\{\tau_i\}_{i=0}^n$ be the Gabriel filtration in $R$-tors. Suppose that every $\tau_i$ is perfect, then $Gdim(_{\tau_1}R)<n$.
\end{cor}

\begin{proof}
If $\{\omega_j\}$ is the Gabriel filtration in $_{\tau_1}R$ then, by Lemma \ref{gab} $\mathfrak{q}_{\tau_1}(\tau_{i+1})=\omega_i$ for all $0\leq i$.
Since $Gdim(R)=n$ then $\tau_n=\chi$ so $\omega_{n-1}=\mathfrak{q}_{\tau_1}(\tau_n)=\mathfrak{q}_{\tau_1}(\chi)=\chi\in {_{\tau_1}}R$-tors. This implies that $Gdim({_{\tau_1}}R)\leq n-1$.
\end{proof}

\begin{lemma}\label{sup}
Let $\tau\in R$-tors and $\{\sigma_i\}_{i\in I}\subseteq R$-tors be a family of perfect torsion theories such that $\tau\leq\sigma_i$ for all $i\in I$. If $\mathfrak{q}_\tau:R-Mod\to {_\tau R}-Mod$ is the localization functor then 
\[\mathfrak{q}_\tau(\bigvee_{i\in I}\sigma_i)=\bigvee_{i\in I}\mathfrak{q}_\tau(\sigma_i)\]
\end{lemma}

\begin{proof}
Write $\bigvee_{i\in I}\sigma_i=(\mathfrak{T}_{\bigvee \sigma_i}, \mathfrak{F}_{\bigvee \sigma_i})$. Then 
\[\mathfrak{q}_\tau(\bigvee_{i\in I}\sigma_i)=(\mathfrak{q}_\tau(\mathfrak{T}_{\bigvee \sigma_i}),\mathfrak{q}_\tau(\mathfrak{F}_{\bigvee \sigma_i}))\]
The torsion free class of $\bigvee_{i\in I}\sigma_i$ is described as $\mathfrak{F}_{\bigvee\sigma_i}=\bigcap_{i\in I}\mathfrak{F}_{\sigma_i}$. So, if $\mathfrak{q}_\tau(N)\in\mathfrak{q}_\tau(\mathfrak{F}_{\bigvee\sigma_i})$ then $N\in \bigcap_{i\in I}\mathfrak{F}_{\sigma_i}$, hence $\mathfrak{q}_\tau(N)\in\bigcap_{i\in I}\mathfrak{q}_\tau(\mathfrak{F}_{\sigma_i})$. Thus
\[\bigvee_{i\in I}\mathfrak{q}_\tau(\sigma_i)\leq \mathfrak{q}_\tau(\bigvee_{i\in I}\sigma_i)\]
Now, suppose $\bigvee_{i\in I}\mathfrak{q}_\tau(\sigma_i)< \mathfrak{q}_\tau(\bigvee_{i\in I}\sigma_i)$, that is, $\mathfrak{T}_{\bigvee\mathfrak{q}_\tau(\sigma_i)}<\mathfrak{q}_\tau(\mathfrak{T}_{\bigvee \sigma_i})$ then there exists $0\neq \mathfrak{q}_\tau(N)\in \mathfrak{q}_\tau(\mathfrak{T}_{\bigvee \sigma_i})$ such that $\mathfrak{q}_\tau(N)\in\mathfrak{F}_{\bigvee\mathfrak{q}_\tau(\sigma_i)}=\bigcap_{i\in I}\mathfrak{F}_{\mathfrak{q}_\tau(\sigma_i)}$.

Since $\tau(N/\tau(N))=0$ then $\mathfrak{q}_\tau(N)=\mathfrak{q}_\tau(N/\tau(N))$, we have that $N\in\mathfrak{T}_{\bigvee \sigma_i}$ then $N/\tau(N)\in\mathfrak{T}_{\bigvee \sigma_i}$. Therefore, we can assume $N$ is $\tau$-torsion free. By the choice of $N$ there exists $j\in I$ such that $N\notin\mathfrak{F}_{\sigma_j}$, so $\sigma_j(N)=N'\neq 0$. On the other hand, $\mathfrak{q}_\tau(N)\in \bigcap_{i\in I}\mathfrak{F}_{\mathfrak{q}_\tau(\sigma_i)}$ then there exists $N_j\in\mathfrak{F}_{\sigma_j}$ such that $\mathfrak{q}_\tau(N)=\mathfrak{q}_\tau(N_j)$. Since $\tau<\sigma_j$, $N_j$ is $\tau$-torsion free, thus $N_j\leq_e\mathfrak{N_j}=\mathfrak{q}_\tau(N)$. This implies that $0\neq N'\cap N_j$ but $N'$ is $\sigma_j$-torsion and $N_j$ is $\sigma_j$-torsion free, this is a contradiction. Thus 
\[\bigvee_{i\in I}\mathfrak{q}_\tau(\sigma_i)= \mathfrak{q}_\tau(\bigvee_{i\in I}\sigma_i)\]
\end{proof}

\begin{rem}
In general Lemma \ref{gab} is not true for infinite ordinals. Let $R$ be a ring with Gabriel dimension, $Gdim(R)=\alpha$ , $\omega<\alpha$. Let $\{\tau_i\}_{i=0}^\alpha$ the Gabriel filtration in $R$-tors and suppose that every $\tau_i$ is perfect. Then, by the proof of Lemma \ref{gab}, if $\{w_j\}$ is the Gabriel filtration in ${_{\tau_1} R}$-tors we have that $\mathfrak{q}_{\tau_1}(\tau_{i+1})=w_{i}$ for all $i\in\mathbb{N}$. If $\omega$ is the first infinite ordinal, by Lemma \ref{sup}
\[\mathfrak{q}_{\tau_1}(\tau_\omega)=\mathfrak{q}_{\tau_1}(\bigvee_{i\in\mathbb{N}}\tau_i)=\bigvee_{i\in\mathbb{N}}\mathfrak{q}_{\tau_1}(\tau_i)=\bigvee_{i\in\mathbb{N}}w_{i-1}=w_\omega\] 
\end{rem}

\begin{dfn}
An injective left $R$-module $E$ is called completely injective if every factor module of $E$ is injective.
\end{dfn}

The following result is well known, see \cite[18, Ex. 10]{andersonrings}

\begin{prop}
Let $R$ be a ring. $R$ is left hereditary if and only if every injective module is completely injective.
\end{prop}

\begin{prop}\label{loc}
Let $R$ be a left QI-ring and $\tau\in R$-tors a perfect torsion theory. Then the ring of quotients $_\tau R$ is a left QI-ring.
\end{prop}

\begin{proof}
Let $_\tau A$ be a quasi-injective $_\tau R$-module. Then we can consider $A$ as a $\tau$-torsion free $R$-module and by \cite[IX, Proposition 2.5]{stenstromrings} $E(A)$ is an injective envelope of $_\tau A$ in $_\tau R$-Mod. Hence $_\tau A\leq E(A)$ is a fully invariant $_\tau R$-submodule. 

Let $f\in End_R(E(A))$, since $E(A)$ is $\tau$-closed then $f$ is an $_\tau R$-morphism. Then $f(_\tau A)\leq {_\tau A}$, i.e., $_\tau A$ is a quasi-injective $R$-module. Since $R$ is left QI, $_\tau A$ is an injective $R$-module. Thus by \cite[IX, Proposition 2.7]{stenstromrings} $_\tau A$ is an injective $_\tau R$-module. 
\end{proof}

\begin{rem}\label{socg}
Let $R$ be a left QI-ring. Consider the class of all semisimple left $R$-modules, it is known that this class is a hereditary pretorsion class but since $R$ is left QI then, semisimple modules form a hereditary torsion class \cite[Theorem 3.9]{daunsqi}. Let us denote the hereditary torsion theory associated to the semisimple torsion class by $\tau_{ss}$. The radical associated to $\tau_{ss}$ is $Soc$. 

In the same way, if we consider the pretorsion class of all singular modules it is a hereditary torsion class. We will denote the hereditary torsion class by $\tau_g$ and the radical associated by $\mathcal{Z}$. Notice that if $R$ is a simple ring then $\tau_g$ is the unique coatom in $R$-tors.
\end{rem}

\begin{thm}\label{ssperf}
Let $R$ be a (simple) left QI-ring. Suppose that $Gdim(R)=n$ and let $\{\tau_i\}_{i=1}^n$ be the Gabriel filtration in $R$-tors. Suppose $\tau_1$ is perfect. If $_{\tau_1}R$ is left hereditary then $R$ is left hereditary.
\end{thm}

\begin{proof}
Since $R$ is a left QI-ring then  $\tau_1=\tau_{ss}$ and by hypothesis $_{\tau_1}R$ is left hereditary.

Now, let $E$ be an indecomposable non singular injective left $R$-module. Let $E\to F$ an epimorphism. Since $R$ is a left noetherian and left $V$-ring $F=Soc(F)\oplus F'$ where $F'$ is ${\tau_1}$-torsion free and $Soc(F)$ is injective. So, to prove $R$ is left hereditary is enough to prove that every factor module $F$ of $E$ with  $Soc(F)=0$ is injective.

Let $\rho:E\to F$ be an epimorphism such that $Soc(F)=0$. This implies that $Ker(\rho)\in Sat_{\tau_1}(E)$. By \cite[Proposition 4.2]{stenstromrings} $Sat_{\tau_1}(E)$ consist of the ${\tau_1}$-closed submodules of $E$. Consider the following diagram
\[\xymatrix{0\ar[r] & Ker(\rho)\ar[r]\ar[d] & E\ar[r]\ar[d] & F\ar[r]\ar[d]^{_F\psi} & 0 \\ 0\ar[r] & Ker(\rho)\ar[r] & E\ar[r] & _{\tau_1}F\ar[r] & 0	}\]

Since ${\tau_1}$ is perfect the localization functor is exact and $Ker(\rho)$ and $E$ are ${\tau_1}$-closed, so the second row is exact. This implies $F$ is ${\tau_1}$-closed. Thus $\rho$ is an $_{\tau_1}R$-morphism. Since $_{\tau_1}R$ is left hereditary and $E$ is an injective $_{\tau_1}R$-module then $F$ is an injective $_{\tau_1}R$-module. Thus $F$ is an injective $R$-module.

By \cite[Proposition 14A]{faithhereditary} every injective $R$-module is completely injective, thus $R$ is left hereditary.
 
\end{proof}

\begin{thm}\label{qi}
Let $R$ be a (simple) left QI-ring. Suppose that $Gdim(R)=n$ and let $\{\tau_i\}_{i=1}^n$ be the Gabriel filtration in $R$-tors. Then the following conditions are equivalent:
\begin{enumerate}
\item Every $\tau_i$ is perfect
\item $R$ is left hereditary.
\end{enumerate}
\end{thm}

\begin{proof}
$\Rightarrow$ By induction over $n$. 

If $n=1$ then $R$ is semisimple, and thus $R$ is hereditary.

Suppose the result is true for all left QI-rings $R$ with $Gdim(R)<n$ such that every element in the Gabriel filtration is perfect. Let $R$ be a left QI-ring with $Gdim(R)=n$. By hypothesis $\tau_1=\tau_{ss}$ is perfect.

 By Proposition \ref{loc} $_{\tau_1}R$ is a left QI-ring and by Lemma \ref{gab2} $Gdim(_{\tau_1}R)<n$. Since the Gabriel filtration in $_{\tau_1}R$-tors is $\{\mathfrak{q}(\tau_i)|1\leq i <n\}$ then, by Lemma \ref{tt} we can apply the induction hypothesis. Hence $_{\tau_1}R$ is left hereditary. Thus by Theorem \ref{ssperf} $R$ is left hereditary.

$\Leftarrow$. If $R$ is left hereditary and left QI-ring then every hereditary torsion theory is perfect \cite[XI, 
Corollary 3.6]{stenstromrings}.
\end{proof}

\begin{dfn}
An $R$-module $M$ satisfies the restricted left socle condition (RLS) if for any essential submodule $N\neq M$, the factor module $M/N$ has non zero socle.
\end{dfn}

\begin{prop}\label{socr}
Let $R$ be a simple left QI-ring which is non semisimple. The following are equivalent:
\begin{enumerate}
\item $R$ satisfies RLS.
\item $\tau_{ss}=\tau_g$
\item $Gdim(R)=2$.
\item There exists a non singular indecomposable completely injective module $E$ which is $\tau_{ss}$-cocritical.
\end{enumerate}
 
\end{prop}

\begin{proof}
$\textit{1} \Rightarrow \textit{2}$. Since $R$ is non semisimple, then by Remark \ref{socg} $\tau_{ss}\leq\tau_{g}$. Now let $M$ be a singular module and $0\neq m\in M$. Then $(0:m)$ is an essential left ideal of $R$. Thus $0\neq Soc(R/(0:m))=Soc(Rm)$. This implies that $Soc(M)\leq_eM$ but $Soc(M)$ is a direct summand, so $M=Soc(M)$.

$\textit{2}\Rightarrow\textit{3}$. If $\{\tau_i|i\geq 0\}$ is the Gabriel filtration in $R$-tors then $\tau_1=\tau_{ss}=\tau_g$. Since $R$ is simple $\tau_g$ is a coatom in $R$-tors, thus $\tau_2=\chi$. 

$\textit{3}\Rightarrow\textit{4}$. If $Gdim(R)=2$, then the Gabriel filtration in $R$-tors is $\{\xi,\tau_{ss}, \chi\}$. Therefore
\[\chi=\tau_{ss}\vee\bigvee\{\xi(N)|N\;\tau_{ss}\text{-cocritical}\}\]
Assume that every $\tau_{ss}$-cocritical module is singular, then $\chi\leq\tau_g$, this is a contradiction. Hence, there exists a non singular $\tau_{ss}$-cocritical module $N$. Since $N$ is cocritical, it is uniform and so $E(N)$ is a non singular indecomposable injective module. By \cite[proposition 2.1]{golancocritically}   $E(N)$ is also $\tau_{ss}$-cocritical, thus we are done.

$\textit{4}\Rightarrow\textit{1}$. Since $R$ is a simple ring then $R$ is a prime ring. By \cite[Corollary 2.15]{mauacc}, $E$ is up to isomorphism the only non singular indecomposable injective module . Now, by \cite[Theorem 2.20]{mauacc} $E(R)\cong E^k$ for some $k>0$; since $E$ is completely injective then $E(R)$ so does by \cite[Proposition 14A]{faithhereditary}. Let $I\leq_eR$, then $I\leq_eE(R)$ and so $E(R)/I$ is semisimple. Thus $R/I$ is semisimple.


\end{proof}

\begin{rem}
In \cite[Theorem 17]{faithhereditary} it is constructed an indecomposable injective module $E$ such that $E$ is non semisimple and satisfies $RLS$. Then $Soc(E)=0$, that is, $E$ is $\tau_{ss}$-torsion free. Since $E$ is uniform and satisfies $RLS$ then $E$ is $\tau_{ss}$-cocritical. Thus if $E$ is nonsingular then $E$ satisfies condition \textit{4} of Proposition \ref{socr}. 
\end{rem}

As Corollary we have the next result due to C. Faith \cite[Theorem 18]{faithhereditary} 

\begin{cor}
Any left QI-ring $R$ with restricted left socle condition is left hereditary.
\end{cor}

\begin{proof}
By Proposition \ref{socr} $R$ has $Gdim(R)=2$. Hence the Gabriel filtration in $R$-tors is $\{\xi,\tau_g, \chi\}$ where $\xi$ and $\chi$ are the least and greatest elements of $R$-tors respectively. The element $\tau_g\in R$-tors is the Goldie's torsion theory and it is perfect because $R$ is left noetherian \cite[Proposition 2.12 and Proposition 3.4]{stenstromrings}. Thus, by Theorem \ref{qi} $R$ is left hereditary.
\end{proof}

\begin{lemma}\label{prodrls}
Let $R=R_1\times\cdots\times R_n$ be a finite product of rings. Then $R$ satisfies $RLS$ if and only if $R_i$ satisfies RLS for all $1\leq i\leq n$. 
\end{lemma}

\begin{proof}
$\Rightarrow$. Let $1\leq i\leq n$ and $I_i$ an essential left ideal of $R_i$. Then $I=R_1\times\cdots I_i\times\cdots R_n$ is an essential left ideal of $R$. By hypothesis, $R/I$ contains a simple $R$-module $S$, and we have that $R/I\cong R_i/I_i$. Thus $S$ is a simple $R_i$-module and it can be embedded in $R_i/I_i$.

$\Leftarrow$. Let $I$ be an essential left ideal of $R$, then $I=I_1\times\cdots\times I_n$ with $I_i$ an essential left ideal of $R_i$. By hypothesis $R_i/I_i$ contains a simple $R_i$-module and we have that $R/I\cong (R_1/I_1)\oplus\cdots\oplus (R_n/I_n)$. Thus $R/I$ contains a simple $R$-module. 
\end{proof}

\begin{rem}\label{prodinj}
Let $R=R_1\times\cdots\times$ be a product of rings. Notice that $E$ is a non singular indecomposable injective $R$-module then $E$ is a non singular indecomposable injective $R_i$-module for some $1\leq i\leq n$. On the other hand, if $E_i$ is a non singular indecomposable injective $R_i$-module then $E_i$ is a non singular indecomposable injective $R$-module.
\end{rem}

\begin{thm}
Let $R=R_1\times\cdots\times R_n$ be a left QI-ring such that each $R_i$ is a simple left $QI$-ring and non semisimple. The following conditions are equivalent:
\begin{enumerate}
\item $R$ satisfies $RLS$.
\item For each $1\leq i\leq n$ there exists a non singular indecomposable injective $R_i$-module $E_i$ 
which are $\tau_{ss}$-cocritical as $R$-modules.
\item $Gdim(R)=2$. 
\item $\tau_{ss}=\tau_g$ in $R$-tors.

\end{enumerate}
\end{thm}

\begin{proof}
$\textit{1}\Rightarrow\textit{2}$. Since $R$ satisfies $RLS$ then by Lemma \ref{prodrls} each $R_i$ satisfies $RLS$. Hence by Proposition \ref{socr} there exist a non singular indecomposable injective $R_i$-module $E_i$ which satisfies $RLS$. 

$\textit{2}\Rightarrow\textit{3}$. Let $\{\tau_j\}$ be the Gabriel filtration in $R$-tors. By Remark \ref{prodinj} each $E_i$ is a non singular indecomposable injective $R$-module, so each $E_i$ is $\tau_{ss}$-torsion free. Since each $E_i$ is $\tau_{ss}$-cocritical then 
\[\tau_{ss}\vee \bigvee\xi(E_i)\leq \tau_2\]

Now, if $E$ is a non singular indecomposable injective $R$-module then $E$ is a non singular indecomposable injective $R_i$-module for some $1\leq i\leq n$. But since $R_i$ is simple and hence a prime ring, by \cite[Corollary 2.20]{mauacc} $E\cong E_i$. Thus all non singular indecomposable injective $R$-modules, up to isomorphism, are $E_1,...,E_n$. Again by \cite[Corollary 2.20]{mauacc} $\widehat{R}\cong E_1^{k_1}\oplus\cdots\oplus E_n^{k_n}$ where $\widehat{R}$ denotes the injective hull of $R$ for some natural numbers $k_1,...,k_n$. Thus $\tau_{ss}\vee\bigvee\xi(E_i)=\chi$. So, $Gdim(R)=2$.

$\textit{3}\Rightarrow\textit{4}$. If $Gdim(R)=2$ then the Gabriel filtration in $R$-tors is $\{\xi,\tau_{ss},\chi\}$. Since every $R_i$ is a simple left QI-ring non semisimple, then all simple $R$-modules are singular. Thus $\tau_{ss}\leq\tau_g$. 

Suppose $\tau_{ss}<\tau_g$ then there exists $C$ such that $C$ is $\tau_{ss}$-cocritical and $\tau_g$-torsion. If $c\in C$, $ann(c)\leq_eR$, and $ann(c)=I_1\times\cdots\times I_n$ with $I_i\leq_e R_i$. Hence $R_I/I_i$ is singular and we have a monomorphism 
\[R/I\cong(R_1/I_1)\oplus\cdots\oplus(R_n/I_n)\to C\]

Since every $R_i$ is a simple left QI-ring and $R_i/I_i$ is singular then by Proposition \ref{socr} $R_i/I_i$ is a semisimple $R_i$ module, hence it is semisimple as $R$-module. Thus $C$ is semisimple and $\tau_{ss}=\tau_g$.

$\textit{4}\Rightarrow\textit{1}$. Let $I\leq_eR$, then $R/I$ is $\tau_g$-torsion then it is $\tau_{ss}$-torsion. This implies that $R/I$ contains a simple $R$-module. Thus $R$ satisfies $RLS$.
\end{proof}

\bibliographystyle{plain}
\bibliography{biblio}

\begin{thebibliography}{10}

\bibitem{andersonrings}
Frank~W Anderson and Kent~R Fuller.
\newblock {\em Rings and categories of modules}.
\newblock Springer-Verlag, 1992.

\bibitem{boylehereditary}
A~Boyle.
\newblock Hereditary {QI}-rings.
\newblock {\em Transactions of the American Mathematical Society}, 192, 1974.

\bibitem{mauacc}
Jaime Castro, Mauricio Medina, and Jos{\'{e}} R{\'{i}}os.
\newblock Modules with ascending chain condition on annihilators and goldie
  modules.
\newblock to appear.

\bibitem{cozzenshomological}
John~H Cozzens.
\newblock Homological properties of the ring of differential polynomials.
\newblock {\em Bulletin of the American Mathematical Society}, 76(1):75--79,
  1970.

\bibitem{daunsqi}
John Dauns and Yiqiang Zhou.
\newblock {QI}-modules.
\newblock In {\em Modules and Comodules. Trends in Mathematics}, pages
  173--183. Birkh\"{a}user Verlag Basel, 2008.

\bibitem{faithhereditary}
Carl Faith.
\newblock On hereditary rings and {B}oyle's conjecture.
\newblock {\em Archiv der Mathematik}, 27(1):113--119, 1976.

\bibitem{golantorsion}
Jonathan~S. Golan.
\newblock {\em Torsion theories}, volume~29.
\newblock Longman Scientific and Technical, Essex, and John Wiley and Sons, New
  York, 1986.

\bibitem{golancocritically}
Jonathan~S. Golan and Zoltan Papp.
\newblock Cocritically nice rings and boyle's conjecture.
\newblock {\em Communications in Algebra}, 8(18):1775--1798, 1980.

\bibitem{rizvianapproach}
Dinh~Van Huynh and S.~Tariq Rizvi.
\newblock An approach to boyle's conjecture.
\newblock {\em Proceedings of the Edinburgh Mathematical Society},
  (40):267--273, 1997.

\bibitem{kaschmodules}
F~Kasch.
\newblock {\em Modules and rings}, volume~17.
\newblock Academic Press, 1982.

\bibitem{lamlectures}
T.~Y. Lam.
\newblock {\em Lectures on modules and rings}.
\newblock Springer, 1999.

\bibitem{stenstromrings}
Bo~Stenstr{\"o}m.
\newblock {\em Rings of quotients: An introduction to methods of ring theory}.
\newblock Springer-Verlag, 1975.

\end{thebibliography}
\end{document}